\documentclass[final,1p,times]{elsarticle}
\usepackage{epsfig}
\usepackage{tabularx}
\usepackage{graphicx}
\usepackage{amsthm}
\usepackage{amsmath}
\usepackage{amsfonts}
\usepackage{euscript}
\usepackage{color}
\usepackage{epstopdf}

\newcommand{\bigo}[1]{\mathcal{O}\left( #1 \right) }

\newcommand{\no}{\noindent}

\newtheorem{thm}{Theorem}
\newtheorem{lemma}{Lemma}
\newtheorem{cor}{Corollary}
\newtheorem{define}{Definition}
\newcommand{\OO}{\mathcal{O}}

\theoremstyle{remark}
\newtheorem{rem}{Remark}
\def\l{\left}
\def\r{\right}
\def\R#1{$(\ref{#1})$}
\newcommand{\bb}[1]{\begin{equation}\label{#1}}
\newcommand{\ee}{\end{equation}}
\newcommand{\bbb}{\begin{eqnarray}}
\newcommand{\eee}{\end{eqnarray}}
\newcommand{\bbbb}{\begin{eqnarray*}}
\newcommand{\eeee}{\end{eqnarray*}}
\newcommand{\nnn}{\nonumber}
\everymath{\displaystyle}
\newcommand{\x}{\textbf{x}}
\newcommand{\w}{\textbf{w}}
\newcommand{\vv}{\textbf{v}}
\newcommand{\f}{\textbf{f}}
\newcommand{\z}{\textbf{z}}
\newcommand{\xx}{\boldsymbol{\xi}}

\newcommand{\e}{\boldsymbol{\varepsilon}}

\newcommand{\RR}{\mathbb{R}}

\newcommand{\clearallnum}{
    \numberwithin{equation}{section} \setcounter{equation}{0}
    \numberwithin{thm}{section} \setcounter{thm}{0}
    \numberwithin{lemma}{section} \setcounter{lemma}{0}
    \numberwithin{cor}{section} \setcounter{cor}{0}
    \numberwithin{rem}{section} \setcounter{rem}{0}
    \numberwithin{define}{section} \setcounter{define}{0}}

\begin{document}


\begin{frontmatter}

\title{A Nonlinear Splitting Algorithm for Systems of Partial Differential Equations with self-Diffusion}
\author{Matthew A. Beauregard$^{a,}$\footnote{This author was supported in part by internal research grant (No. 107552-26423-150) from Stephen F. Austin State University.}, Joshua Padgett$^b$, and Rana Parshad$^{c}$}
\address{$^a$Department of Mathematics and Statistics, Stephen F. Austin State University, Nacogdoches, TX, 75962
\\
$^b$Department of Mathematics, Baylor University, Waco, TX, 76798
\\
$^c$Department of Mathematics, Clarkson University, Potsdam, NY, 13699}

\begin{abstract}
Systems of reaction-diffusion equations are commonly used in biological models of food chains.  The populations and their complicated interactions present numerous challenges in theory and in numerical approximation.  In particular, self-diffusion is a nonlinear term that models overcrowding of a particular species.  The nonlinearity complicates attempts to construct efficient and accurate numerical approximations of the underlying systems of equations.  In this paper, a new nonlinear splitting algorithm is designed for a partial differential equation that incorporates self-diffusion. We present a general model that incorporates self-diffusion and develop a numerical approximation. The numerical analysis of the approximation provides criteria for stability and convergence. Numerical examples are used to illustrate the theoretical results.
\end{abstract}

\end{frontmatter}


\bibliographystyle{abbrv}

\section{Introduction} \clearallnum

This paper is motivated by a three-species food chain model first developed in \cite{U97} and analyzed in \cite{AA02}. Recently, this model was improved to consider overcrowding effects of the population species in \cite{BeauPar2015}.  Our goal of this paper is to develop reliable, accurate, efficient, and valid numerical approximations that incorporate the nonlinear overcrowding term for the top predator.

Consider an invasive species $r$ that has invaded a certain two dimensional habitat.  Let $r$ predate on a middle predator $v$, which in turn predates on a prey $u$. A partial differential equation that includes overcrowding is,
\begin{eqnarray}
\label{eq:(1.1)}
\partial _{t}r &=& d_3 \Delta r + d_4 \Delta r^2  +cr^{2}-w _{3}\frac{r^{2}}{v+D_{3}} \equiv d_3 \Delta r + d_4 \Delta r^2 + h(u,v,r),\\
\label{eq:(1.2)}
\partial _{t}v &=& d_{2}\Delta v-a_{2}v+w_{1}\frac{uv}{u+D_{1}}-w_{2}\frac{vr}{v+D_{2}} \equiv d_{2}\Delta v + g(u,v,r),\\
\label{eq:(1.3)}
\partial _{t}u &=& d_{1}\Delta u+a_{1}u-b_{2}u^{2}-w_{0}\frac{uv}{u+D_{0}}\equiv d_{1}\Delta u + f(u,v,r),
\end{eqnarray}
\noindent defined on $\mathbb{R}^{+}\times \Omega$. Here $\Omega \subset \mathbb{R}^{2}$ and $\Delta$ is the two dimensional Laplacian operator.  We define $\x$ to be the spatial coordinate vector in two dimensions. The parameters $d_{1}$, $d_{2}$ and $d_{3}$ are positive diffusion coefficients.
The initial populations are given as
\[ u(0,\x)=u_{0}(\x),~~ v(0,\x)=v_{0}(\x),~~ r(0,\x)=r_{0}(\x) ~~~~~ \x \in \Omega, \]
\noindent are assumed to be nonnegative and uniformly bounded on $\Omega$. Appropriate boundary conditions are specified.  Here, we examine Dirichlet boundary conditions, however our analysis extends to the Neumann boundary condition case in a straight forward manner.
The parameter definitions are given in Table \ref{tab:table1}:
\begin{table}[h]
    \begin{tabular}{ll}\hline
      Symbols & Meaning \\ \hline
$u$                      & Prey\\
$v$                      & Middle Predator\\
$r$                       & Top Predator \\
$a_1$                  & Growth rate of prey $u$ \\
$a_2$                  & Measures the rate at which $v$ dies out when there is no $u$ to prey on and no $r$\\
$w_{i}'s$              & Maximum value that the per-capita rate can attain\\
$D_0$, $D_1$     & Measure the level of protection provided by the environment to the prey\\
 $b_{2}$               & Measure of the competition among prey, $u$\\
 $D_2$                 & Value of $v$ at which its per capita removal rate becomes ${w_2}/2$\\
 $D_3$                 & Loss in $r$ due to the lack of its favorite food, $v$\\
 $c$                   & Growth rate of $r$ via sexual reproduction\\
 $d_4$                 & The strength of the overcrowding term \\
 \hline
    \end{tabular}
    \caption{List of parameters used in the three species food chain model. All these parameters are positive constants.}
    \label{tab:table1}
\end{table}

This model is rich in dynamics and stems from the Leslie-Gower formulation \cite{L48}, that is, the middle predator is depredated at a Holling type II rate, and the generalist top predator grows logistically as $cr$, and loses due to intraspecies competition as $-w_{3}r^{2}/(v+D_{3})$. The literature is abundant with investigations of variants to this model [2,~3,~7,~10,~12-17,~19,~21]. 
However, the development and analysis of accurate numerical approximations has not been considered, especially in situations involving the overcrowding term.  The overcrowding term can be viewed as a severe penalty to crowding in the top predator forcing a strong movement to lower concentrations of $r$.

While the above model motivates this paper we develop a nonlinear algorithm for
\bb{model}
u_t = \Delta\left( u + u^2 \right) + f(u,\x,t),
\ee
where $\Delta$ is the standard $2-$dimensional Laplacian, $f(u,\x,t)$ is a nonlinear reactive term, and appropriate initial and boundary conditions are given for $u(\x,t)$.

This paper is organized as follows. In section 2 we present the nonlinear variable time splitting model which is based on a modified Douglass-Gunn splitting method.  Section 3 details our numerical analysis of the proposed algorithm.  It is shown that the method is stable and second-order convergent in time under reasonable criteria for the temporal and spatial sizes.  Section 4 contains examples that illustrate our theoretical results and explores the dynamics of \R{eq:(1.1)}-\R{eq:(1.3)}, in particular the effect of the self-diffusion and overcrowding on the numerical solution. Section 5 summarizes our key results.

In the ensuing discussion all lowercase bold letters indicate vectors,
uppercase letters are used for matrices.  The $\ell^2$-norm is used
throughout discussions unless otherwise specified. That is, given a vector $\x \in \mathbb{R}^n$, then
\[ \| \x \| = \sqrt{\sum_{i=1}^n |x_i|^2}. \]
The matrix norms considered will be the spectral norm, which is induced by the above vector norm. 

We define a scheme as \textit{computationally efficient} if it is \textit{second order accurate} in space and time or better
and the \textit{number of operations per time step is directly proportional to the number of unknowns}.

\section{Nonlinear Model}

\no We consider the following model
\begin{eqnarray}\label{nonlinearmodel}
u_t = \Delta\left( u + u^2 \right) + f(u,\x,t),
\end{eqnarray}
where $\Delta$ is, in this case, the $2-$dimensional Laplacian, $\x = (x,y)$, and appropriate initial conditions are given. Dirichlet boundary conditions are assumed. Without loss of generality, we assume a square domain $\Omega = (0,1)\times (0,1)$ in the following discussions.

Given $N\gg 0,$ we may inscribe over $\Omega$ the mesh $\mathcal{D}_h = \{(x_i,y_j)~|~i,j=0,1,\dots,N+1\},$ where $h=1/(N+1)$ and $x_i=ih$ and $y_i=jh$ for $i,j=0,1,\dots,N+1.$ Further, we define $u_{i,j}(t)$ as the approximation to the exact solution $u(x_i, y_j, t)$ and let $\vv=(u_{1,1}(t), u_{2,1}(t), \ldots, u_{N,1}(t), \ldots, u_{N,N}(t))^{\top}.$ Similarly, let $\f=(f(u(x_1,y_1,t),x_1,y_1,t), f(u(x_2,y_2,t),x_2,y_1,t), \ldots, f(u(x_N,y_N,t),x_N, y_N, t))^{\top}$.  We propose the following semidiscretized scheme to approximate \R{nonlinearmodel},

\[ \frac{d\vv}{dt} = \left(P + R + P D(\vv) + R D(\vv)\right) \vv + \f, \]

\noindent where $P$, $R$, and $D(\vv)$ are $N^2 \times N^2$ matrices defined as $P = I_N \otimes T,$ $R = T \otimes I_N$, and $D(\vv) = \mbox{diag}(\vv) =  \mbox{diag}(u_{1,1}(t),\dots,u_{N,N}(t)).$ $I_N$ is the $N\times N$ identity matrix and $T$ is the symmetric, tridiagonal $N\times N$ matrix with $1/h^2$ as the lower and upper diagonals and $-2/h^2$ as the main diagonal. Here, we have used second-order finite differences to approximate the spatial derivatives, however other suitable approximates may be incorporated.  However, the theoretical requirements that will ensure stability of our algorithm will need to be resolved to reflect a different spatial discretization. Nevertheless, these changes in the spatial discretization will not affect the temporal advancement developed in this paper. A variable time-step second order Crank-Nicolson method is used to advance the solution in time, namely,
\begin{eqnarray*}
 \left(I - \frac{\tau_k}{2}\left( P + R + P D_{k+1} + R D_{n+1} \right)\right) \vv_{k+1} &=&  \left(I + \frac{\tau_k}{2} \left( P + R + P D_k + R D_k \right)\right) \vv_k \\
 &&~+ \frac{\tau_k}{2} \left( \f_{k+1} + \f_k \right) + \OO\l(\tau_k^3\r),
 \end{eqnarray*}
where $\tau_k$ is the variable temporal step, $D_k = D(\vv_k),$ and $\f_k$ is the vector $\f$ evaluated at time $t_k,$ where $t_k=\sum_{i=0}^{k-1}\tau_i.$ The above can then be factored as,
\begin{eqnarray}
&&\left(I - \frac{\tau_k}{2} P \right)\left(I - \frac{\tau_k}{2} R \right)\left(I - \frac{\tau_k}{2} P D_{k+1} \right)\left(I - \frac{\tau_k}{2} R D_{k+1} \right)\vv_{k+1}\nnn\\
&&~~~ = \left(I + \frac{\tau_k}{2} P \right)\left(I + \frac{\tau_k}{2} R \right)\left(I + \frac{\tau_k}{2} P D_{k} \right)\left(I + \frac{\tau_k}{2} R D_{k} \right)\vv_{k}+\frac{\tau_k}{2}(\f_{k+1}+\f_k)+\OO\l(\tau^3_k\r).~~\label{factorednonlinear}
\end{eqnarray}
This is to be solved using our modified variable time Douglass-Gunn splitting method, namely,
\begin{subequations}
\begin{eqnarray}
\label{StepOne}   \l(I - \frac{\tau_k}{2} P\r)\tilde{\vv}^{(1)} &=& \left(I + \frac{\tau_k}{2} \left(P + 2R + 2PD_k + 2RD_k\right)\right) \vv_k + \tau_k \f_k, \\
\label{StepTwo}   \l(I - \frac{\tau_k}{2} R\r)\tilde{\vv}^{(2)} &=& \tilde{\vv}^{(1)} - \frac{\tau_k}{2} R \vv_k,\\
\label{StepThree} \l(I - \frac{\tau_k}{2} PD_{k+1}\r)\tilde{\vv}^{(3)} &=& \tilde{\vv}^{(2)} - \frac{\tau_k}{2} P D_k \vv_k,\\
\label{StepFour}  \l(I - \frac{\tau_k}{2} RD_{k+1}\r)\vv_{k+1} &=& \tilde{\vv}^{(3)} - \frac{\tau_k}{2} R D_k \vv_k + \frac{\tau_k}{2}\left( \f_{k+1} - \f_k\right)
\end{eqnarray}
\end{subequations}
The third and fourth steps involve the implicit term $D_{k+1}$. This can be approximated by taking an Euler step, that is, $D_{k+1} = \mbox{diag}(\vv_{k+1}) = D_k + \tau_k~ \mbox{diag}\l( (P+R+PD_k + RD_k) \vv_k + \f_k\r) + \OO\l(\tau_k^2\r).$ Furthermore, an Euler step is used to evaluate $\f_{k+1}$.  This approximation maintains the second order accuracy of the splitting method.  As will be shown in the following section, the variable time step is necessary to ensure stability of the method.

This splitting method can be shown to be equivalent to \R{factorednonlinear}.  To see this, first multiply \R{StepFour} by
$$\l(1-\frac{\tau_k}{2} P\r)\l(1-\frac{\tau_k}{2} R\r)\l(1-\frac{\tau_k}{2} PD_{k+1}\r)$$
to obtain,
\begin{eqnarray*}
&& \left(I - \frac{\tau_k}{2} P \right)\left(I - \frac{\tau_k}{2} R \right)\left(I - \frac{\tau_k}{2} P D_{k+1} \right)\left(I - \frac{\tau_k}{2} R D_{k+1} \right)\vv_{k+1} \\
&&~~~ = \left(I - \frac{\tau_k}{2} P \right)\left(I - \frac{\tau_k}{2} R \right)\left(I - \frac{\tau_k}{2} P D_{k+1} \right)\left(\tilde{\vv}^{(3)} - \frac{\tau_k}{2} R D_k \vv_k + \frac{\tau_k}{2}\left( \f_{k+1} - \f_k\right)\right) \\
&&~~~ = \left[ 1 + \frac{\tau_k}{2}(P + R + P D_k + R D_k)  + \frac{\tau_k^2}{4} (PR + PPD_k + R PD_k + P RD_k + R R D_k + P D_{k+1} R D_k)\right] \vv_k~~~~~~\\
&&~~~~~~~~~~+\frac{\tau_k}{2}\left(\f_k+\f_{k+1}\right) -\frac{\tau_k^2}{2}\left(P+R+PD_{k+1}\right)\left( \f_{k+1} - \f_k\right) + \bigo{\tau_k^3}\\
&&~~~ = \left[ 1 + \frac{\tau_k}{2}(P + R + P D_k + R D_k) + \frac{\tau_k^2}{4} (PR + PPD_k + R PD_k + P RD_k + R R D_k + P D_{k} R D_k)\right] \vv_k \\
&&~~~~~~~~~~ \frac{\tau_k}{2}\left(\f_k+\f_{k+1}\right) + \bigo{\tau_k^3}\\
&&~~~ = \left(I + \frac{\tau_k}{2} P \right)\left(I + \frac{\tau_k}{2} R \right)\left(I + \frac{\tau_k}{2} P D_{k} \right)\left(I + \frac{\tau_k}{2} R D_{k} \right)\vv_{k} + \frac{\tau_k}{2}\left(\f_k+\f_{k+1}\right) + \bigo{\tau_k^3}.
\end{eqnarray*}

\section{Numerical Analysis} \clearallnum

The numerical solution of the food-chain model \R{eq:(1.1)}-\R{eq:(1.3)} must remain nonnegative for initial nonnegative datum. Hence, the same requirement is in place for our alternative nonlinear model in \R{nonlinearmodel}. We first provide the criteria to ensure nonnegativity of the numerical solution and then use these criteria in the ensuing discussions. To this end, we have the following results.\\

\begin{lemma}
$\|T\| \le \frac{4}{h^2}.$\\
\end{lemma}
\begin{proof}
This is a standard property of the matrix $T,$ which has eigenvalues $\lambda_j = - \frac{4}{h^2} \sin^2\left(\dfrac{\pi j}{2 (N+1)}\right),$ $j=1,\dots,N.$
\end{proof}

\begin{thm} If
\bb{cfl}
\frac{\tau_k}{h^2} < \frac{1}{2\max\{1,\max_{i=1,\dots,N^2}\{(\vv_k)_i\}\}},
\ee
then the matrices
$$I-\frac{\tau_k}{2}P,~~I-\frac{\tau_k}{2}R,~~I+\frac{\tau_k}{2}P,~~I+\frac{\tau_k}{2}R,~~I-\frac{\tau_k}{2}PD_{k+1},~~I-\frac{\tau_k}{2}RD_{k+1},~~I+\frac{\tau_k}{2}PD_k,~~I+\frac{\tau_k}{2}RD_k$$
are nonsingular. Also, $I-\frac{\tau_k}{2}P,~I-\frac{\tau_k}{2}R,~I-\frac{\tau_k}{2}PD_{k+1},~I-\frac{\tau_k}{2}RD_{k+1}$ are inverse positive and $I+\frac{\tau_k}{2}P,~I+\frac{\tau_k}{2}R,~I+\frac{\tau_k}{2}PD_k,~I+\frac{\tau_k}{2}RD_k$ are nonnegative.
\end{thm}
\begin{proof}
Using the previous lemma, we have
\bbbb
\l\|\frac{\tau_k}{2}P\r\| & = & \frac{\tau_k}{2}\|I_n\otimes T\| ~~=~~  \frac{\tau_k}{2}\|T\|\\
& \leq & \frac{\tau_k}{2}\times \frac{4}{h^2} ~~<~~ 1.
\eeee
Hence, $I+\frac{\tau_k}{2}P$ is nonsingular and nonnegative. A similar argument shows that $I+\frac{\tau_k}{2}R$ is nonsingular and nonnegative. Next, consider
\bbbb
\l\|\frac{\tau_k}{2}PD_{k}\r\| & = & \frac{\tau_k}{2}\|(I_N\otimes T)D_{k}\| ~~\le~~ \frac{\tau_k}{2}\|T\|\|D_{k}\|\\
& \le & \frac{\tau_k}{2} \times \frac{4}{h^2} \times \|\vv_{k}\| ~~<~~ 1.
\eeee
Hence, $I+\frac{\tau_k}{2}PD_{k}$ is nonsingular and nonnegative. A similar argument shows that $I+\frac{\tau_k}{2}RD_{k}$ is nonsingular and nonnegative.

Now consider the matrix $M = I-\frac{\tau_k}{2}P.$ Since $M_{ij}\le 0$ for $i\neq j$ and the weak row sum criterion is satisfied, we have that $M^{-1}$ exists and that $M$ is inverse positive \cite{BeauregardJCAM,Henrici}. A similar argument gives that $I-\frac{\tau_k}{2}R$ is nonsingular and inverse positive. Next, we consider the matrix
$$I-\frac{\tau_k}{2}PD_{k+1} = I-\frac{\tau_k}{2}PD_{k+1} - \frac{\tau_k^2}{2}\mbox{diag}\l((P+R+PD_k+RD_k)\vv_k\r)+\OO\l(\tau_k^3\r).$$
Due to the accuracy of our scheme, we only need to consider the matrix
$$N = I-\frac{\tau_k}{2}PD_{k+1} - \frac{\tau_k^2}{2}\mbox{diag}\l((P+R+PD_k+RD_k)\vv_k\r).$$
Note that $N_{i,j}\le 0$ for $i\neq j$ since the off-diagonal elements of $I-\frac{\tau_k}{2}PD_k$ are nonpositive. Further, since we are concerned with the positivity of the solution, we assume that $\vv_k\ge 0.$ Now consider
$$\sum_{j=1}^{N^2}N_{i,j} = 1 - \frac{\tau_k}{2h^2}\l((\vv_k)_{j-1}-2(\vv_k)_j + (\vv_k)_{j+1}\r) - \frac{\tau_k^2}{2}\mbox{diag}\l((P+R+PD_k+RD_k)\vv_k\r)_j,$$
for $i=1,\dots,N^2$ and where $(\vv_k)_0 = (\vv_k)_{N^2+1} = 0.$ We need to show that the weak row sum criterion is satisfied, thus,
\bbbb
\sum_{j=1}^{N^2}N_{i,j} & = & 1 - \frac{\tau_k}{2h^2}\l((\vv_k)_{j-1}-2(\vv_k)_j + (\vv_k)_{j+1}\r) - \frac{\tau_k^2}{2}\mbox{diag}\l((P+R+PD_k+RD_k)\vv_k\r)_j\\
& = & 1 + \frac{\tau_k}{h^2}(\vv_k)_j - \frac{\tau_k}{2h^2}\l((\vv_k)_{j-1}+(\vv_k)_{j+1}\r) - \frac{\tau_k^2}{2}\mbox{diag}\l((P+R+PD_k+RD_k)\vv_k\r)_j\\
& > & \frac{1}{2} + \frac{\tau_k}{h^2}(\vv_k)_j - \frac{\tau_k^2}{2}\mbox{diag}\l((P+R+PD_k+RD_k)\vv_k\r)_j\\
& > & \frac{1}{2} + \frac{\tau_k}{h^2}(\vv_k)_j - \frac{\tau_k^2}{2}\mbox{diag}\l((\|P\|+\|R\|+\|P\|\|D_k\|+\|R\|\|D_k\|)\vv_k\r)_j\\
& \ge & \frac{1}{2} + \frac{\tau_k}{h^2}(\vv_k)_j - 4\tau_k\l[\frac{\tau_k}{h^2}\l(1+\max\{(\vv_k)_j\}\r)\r](\vv_k)_j\\
& > & \frac{1}{2} + \frac{\tau_k}{h^2}(\vv_k)_j - 4\tau_k(\vv_k)_j ~~>~~ \frac{\tau_k}{h^2}\max\{1,\max\{(\vv_k)_j\}\} + \frac{\tau_k}{h^2}(\vv_k)_j - 4\tau_k(\vv_k)_j\\
& \ge & 2\tau_k\l[\frac{1}{h^2}-2\r](\vv_k)_j ~~\ge~~  0,
\eeee
since $h\le \sqrt{1/2}$ for $N\gg 1.$ Hence, we have that the weak row sum criterion is satisfied, since we have strict inequality for $j=1$ and $N^2.$ A similar argument gives that $I-\frac{\tau_k}{2}RD_{k+1}$ is nonsingular and inverse positive.
\end{proof}

\begin{cor} If
\bb{cfl1}
\frac{\tau_k}{h^2} < \frac{1}{4\max\{1,\max_{i=1,\dots,N^2}\{(\vv_k)_i\}\}},
\ee
then the conditions from Theorem 3.1 still hold, and in addition, the matrices
$$I-\tau_kPD_{k+1},~~I-\tau_kRD_{k+1},~~I+\tau_kPD_k,~~I+\tau_kRD_k$$
are nonsingular. Also, $I-\tau_kPD_{k+1}~\mbox{and}~I-\tau_kRD_{k+1}$ are inverse positive and $I+\tau_kPD_k~\mbox{and}~I+\tau_kRD_k$ are nonnegative.
\end{cor}

In order to prove stability, we now introduce a definition and some lemmas.

\begin{define}
Let $\|\cdot\|$ be an induced matrix norm. Then the
associated logarithmic norm $\mu : \mathbb{C}^{n\times n}\to \mathbb{R}$ of $A\in\mathbb{C}^{n\times n}$ is defined as
$$\mu(A) = \lim_{h\to 0^+} \frac{\|I_n + hA\| - 1}{h},$$
where $I_n\in\mathbb{C}^{n\times n}$ is the identity matrix.
\end{define}

\begin{rem}
When the induced matrix norm being considered is the spectral norm,
then $\mu(A) = \max\left\{\lambda : \lambda\ \text{is an eigenvalue of}\ (A+A^*)/2\right\}.$
\end{rem}

\begin{rem}
If the matrix $A$ negative definite, then we have $\mu(A)\le 0.$
\end{rem}

\begin{lemma}For $\alpha\in\mathbb{C}$ we have
$$\|E(\alpha A)\| \le E(\alpha \mu(A)),$$
where $E(\cdot)$ is the matrix exponential.
\end{lemma}

\begin{proof}
See Golub and Van Loan \cite{Golub}.
\end{proof}

\begin{lemma}
Let \R{cfl1} hold, then $\mu(PD_k),~\mu(RD_k),~\mu(PD_{k+1}),~\mu(RD_{k+1})\le \frac{1}{2}\max\{|\Delta \w_k|\} + ch^2,$ where $\w_k$ is the exact solution to \R{nonlinearmodel} at time $t=t_k$ and $c$ is a positive constant independent of $h.$
\end{lemma}

\begin{proof}
We only need to consider the matrix $PD_k,$ as the other results will follow in a similar manner. By Remark 3.1 we have $\mu(PD_k) = \max\l\{\lambda : \lambda\ \text{is an eigenvalue of}\ (PD_k + (PD_k)^*)/2\r\}.$ By direct calculation we have $PD_k+(PD_k)^* = PD_k + D_kP = \mbox{diag}(X_1,\dots,X_N),$ thus we only need consider the eigenvalues of each block. To that end, we see
$$\l(X_\ell\r)_{i,j} = \l\lbrace\begin{array}{cl}
\l(u_{i-1,\ell}+u_{i,\ell}\r)/2h^2, & i-j=1,\\
-2u_{i,\ell}/h^2, & i-j=0\\
\l(u_{i,\ell}+u_{i+1,\ell}\r)/2h^2, & i-j=-1\\
0, & \mbox{otherwise}.
\end{array}\r.$$
By applying Ger\u{s}chgorin's circle theorem we have
\bbbb
\l|\lambda_i+2u_{i,\ell}/h^2\r| &\le& \l|\l(u_{i-1,\ell}+u_{i,\ell}\r)/2h^2\r| + \l|\l(u_{i,\ell}+u_{i+1,\ell}\r)/2h^2\r|\\
&=& \l(u_{i-1,\ell}+u_{i,\ell}\r)/2h^2 + \l(u_{i,\ell}+u_{i+1,\ell}\r)/2h^2,
\eeee
due to the positivity of the solution. Thus we have
\bbbb
\lambda_i + 2u_{i,\ell}/h^2 & \le & \l(u_{i-1,\ell}+u_{i,\ell}\r)/2h^2 + \l(u_{i,\ell}+u_{i+1,\ell}\r)/2h^2\\
\lambda_i & \le & \frac{u_{i-1,\ell}-2u_{i,\ell}+u_{i+1,\ell}}{2h^2}\\
& = & \frac{1}{2}u_{xx}\bigg|_{(x_i,y_\ell,t_k)}+\OO\l(h^2\r)
\eeee
and by a similar argument
\bbbb
-\l(\frac{1}{2}u_{xx}\bigg|_{(x_i,y_\ell,t_k)}+\OO\l(h^2\r)\r) &\le& \lambda_i.
\eeee
Combining the above results gives $\mu(PD_k)\le \max\{|\Delta\w_k|/2\}+ch^2,$ where $c$ is a positive constant independent of $h,$ which is the desired result. The other bounds follow similarly.
\end{proof}

\begin{lemma}
If \R{cfl1} holds and $\w_k\in W^{2,\infty}(\Omega),$ for $k\ge 0,$ where $\w_k$ is the exact solution to \R{nonlinearmodel} at time $t=t_k,$ then we have
$$\l\|I+\frac{\tau_k}{2}P\r\|,~\l\|I+\frac{\tau_k}{2}R\r\|,~\l\|\l(I-\frac{\tau_k}{2}P\r)^{-1}\r\|,~\l\|\l(I-\frac{\tau_k}{2}R\r)^{-1}\r\|\le 1+\OO\l(\tau_k^2\r)$$
and
$$\l\|I+\frac{\tau_k}{2}PD_k\r\|,~\l\|I+\frac{\tau_k}{2}RD_k\r\|,~\l\|\l(I-\frac{\tau_k}{2}PD_{k+1}\r)^{-1}\r\|,~\l\|\l(I-\frac{\tau_k}{2}RD_{k+1}\r)^{-1}\r\|\le 1+\OO\l(\tau_k\r).$$
\end{lemma}

\begin{proof}
First, we recall that under \R{cfl1} we have the [1/0] Pad{\'e} approximation
\bb{e1}
I+\frac{\tau_k}{2}P = E\l(\frac{\tau_k}{2}P\r) + \OO\l(\tau_k^2\r),
\ee
where $E(\cdot)$ is the matrix exponential. Thus, taking the norm of both side of \R{e1} gives
\bbbb
\l\|I+\frac{\tau_k}{2}P\r\| & = & \l\|E\l(\frac{\tau_k}{2}P\r)+\OO\l(\tau_k^2\r)\r\| ~~\le~~ \l\|E\l(\frac{\tau_k}{2}P\r)\r\|+\OO\l(\tau_k^2\r)\\
& \le & E\l(\frac{\tau_k}{2}\mu(P)\r) + \OO\l(\tau_k^2\r) ~~\le~~ 1 + \OO\l(\tau_k^2\r),
\eeee
where we have appealed to the fact that $P$ is negative definite. A similar argument gives $\l\|I+\frac{\tau_k}{2}R\r\|\le 1+\OO\l(\tau_k^2\r).$ Under \R{cfl1} we have the [0/1] Pad{\'e} approximation
$$\l(I-\frac{\tau_k}{2}P\r)^{-1} = E\l(\frac{\tau_k}{2}P\r)+\OO\l(\tau_k^2\r).$$
So, as above we obtain $\l\|\l(I-\frac{\tau_k}{2}P\r)^{-1}\r\| \le 1+\OO\l(\tau_k^2\r).$ A similar argument gives $\l\|\l(I-\frac{\tau_k}{2}R\r)^{-1}\r\| \le 1+\OO\l(\tau_k^2\r).$ Lemma 3.3 gives that $\mu(PD_k),~\mu(RD_k),~\mu(PD_{k+1}),~\mu(RD_{k+1})\le\frac{1}{2}\max\{|\Delta\w_k|\}+ch^2,$ and hence, similar arguments as above will give the remaining bounds since $\w_k\in W^{2,\infty}(\Omega).$
\end{proof}


First we provide a proof of stability while freezing the nonlinear source term and nonlinear amplification matrices. This is equivalent to assuming that the source term and amplification matrices are constant.

\begin{thm}
Assume that the nonlinear source term and the nonlinear amplification matrices are frozen. If \R{cfl} holds, then the linearized scheme \R{factorednonlinear} is unconditionally
stable in the von Neumann sense, that is,
$$\|\textbf{z}_{k+1}\| \leq c \|\textbf{z}_{0}\|,~~~k\geq 0, $$
where $\textbf{z}_0=\vv_0-\tilde{\vv}_0$ is an initial error,
$\textbf{z}_{k+1}=\vv_{k+1}-\tilde{\vv}_{k+1}$ is the $(k+1)$th perturbed
error vector, and $c>0$ is a constant independent of $k$ and $\tau_k.$
\end{thm}

\begin{proof}
Note that after rearranging \R{factorednonlinear} we have
\bbb
\z_{k+1} & = & \l(I-\frac{\tau_k}{2}RD_{k+1}\r)^{-1}\l(I-\frac{\tau_k}{2}PD_{k+1}\r)^{-1}\l(I-\frac{\tau_k}{2}R\r)^{-1}\l(I-\frac{\tau_k}{2}P\r)^{-1}\nnn\\
&&\times \left(I + \frac{\tau_k}{2} P \right)\left(I + \frac{\tau_k}{2} R \right)\left(I + \frac{\tau_k}{2} P D_{k} \right)\left(I + \frac{\tau_k}{2} R D_{k} \right)\z_{k}.\label{s1}
\eee
Taking the norm of both sides of \R{s1} gives
\bbb
\|\z_{k+1}\| & = & \l\|\l(I-\frac{\tau_k}{2}RD_{k+1}\r)^{-1}\l(I-\frac{\tau_k}{2}PD_{k+1}\r)^{-1}\l(I-\frac{\tau_k}{2}R\r)^{-1}\l(I-\frac{\tau_k}{2}P\r)^{-1}\r.\nnn\\
&&\l.\times \left(I + \frac{\tau_k}{2} P \right)\left(I + \frac{\tau_k}{2} R \right)\left(I + \frac{\tau_k}{2} P D_{k} \right)\left(I + \frac{\tau_k}{2} R D_{k} \right)\z_{k}\r\|\nnn\\
& \le & \l\|\l(I-\frac{\tau_k}{2}RD_{k+1}\r)^{-1}\r\|\l\|\l(I-\frac{\tau_k}{2}PD_{k+1}\r)^{-1}\r\|\l\|\l(I-\frac{\tau_k}{2}R\r)^{-1}\r\|\l\|\l(I-\frac{\tau_k}{2}P\r)^{-1}\r\|\nnn\\
&&\times \l\|\left(I + \frac{\tau_k}{2} P \right)\r\|\l\|\left(I + \frac{\tau_k}{2} R \right)\r\|\l\|\left(I + \frac{\tau_k}{2} P D_{k} \right)\r\|\l\|\left(I + \frac{\tau_k}{2} R D_{k} \right)\r\|\|\z_{k}\|\nnn\\
& \le & \l(1 + c_k\tau_k\r)\|\z_k\|,\label{s3}
\eee
where $c_k$ is a positive constant independent of $\tau_k$ and $h.$ We have also appealed to the bounds from Lemma 3.4 in the last step.
Now using \R{s3} we have
\bb{s4}
\|\z_{k+1}\| ~~\le~~ \l(1+c_k\tau_k\r)\|\z_k\| ~~\le~~ \prod_{i=0}^k (1+c_i\tau_i)\|\z_0\| ~~\le~~ \l(1 + C\sum_{i=0}^k\tau_i\r)\|\z_0\|,
\ee
where $c_0,c_1,\dots,c_k,~C$ are positive constants independent of $\tau_i,~i=0,\dots,k.$ Since we are on a finite time interval, say $[0,T]$ with $T<\infty,$  we may claim that $\sum_{i=0}^k \tau_k \le T.$ Thus, \R{s4} becomes
\bbb
\|\z_k\| ~~\le~~ \l(1 + C\sum_{i=0}^k\tau_i\r)\|\z_0\| ~~\le~~ (1 + CT)\|\z_0\| ~~\le~~ c\|\z_0\|,\nnn
\eee
which gives the desired stability.
\end{proof}

In an effort to show stability and convergence without freezing any of the nonlinear terms, we prove some lemmas which will provide useful bounds.

\begin{lemma}
Assume that \R{cfl1} holds. Let $\w_k$ be the exact solution to \R{nonlinearmodel} and let $\vv_k$ be the solution to \R{factorednonlinear}. If $\xx_k = \vv_k + t^*(\w_k-\vv_k)$ for some $t^*\in (0,1)$ then we have
$$\|PD(\xx_k)\|,\|RD(\xx_k)\| \le \max\l\{|\Delta \w_k|\r\} + ch^2,$$
where $D(\xx_k)=\mbox{diag}(\xx_k)$ and $c$ is a positive constant independent of $h.$
\end{lemma}

\begin{proof}
We first consider $PD(\xx_k).$ Note that $PD(\xx_k) = \mbox{diag}(X_1,\dots,X_N),$ where
$$(X_\ell)_{i,j} = \l\lbrace\begin{array}{cl}
(\xi_{i-1,\ell})_k/h^2, & i-j=1,\\
-2(\xi_{i,\ell})_k/h^2, & i-j=0,\\
(\xi_{i+1,\ell})_k/h^2, & i-j=-1,\\
0, & \mbox{otherwise}.
\end{array}\r.$$
Further, we have $\|PD(\xx_k)\| \le \max_{\ell=1,\dots,N}\|X_\ell\|.$ So, by applying Ger\u{s}chgorin's circle theorem we may obtain the following relation for $\|X_\ell\|$
$$\l|\lambda + 2(\xi_{i,\ell})_k/h^2\r| ~~~\le~~~ \l|(\xi_{i-1,\ell})_k/h^2\r|+\l|(\xi_{i+1,\ell})_k/h^2\r|.$$
Using $\xx_k = \vv_k + t^*(\w_k-\vv_k)$ and that $\vv_k$ is nonnegative under \R{cfl}, we have
\bbb
&&\l|\lambda + \frac{2\l[(1-t^*)(v_{i,\ell})_k + t^*(w_{i,\ell})_k\r]}{h^2}\r| \le \frac{(1-t^*)(v_{i-1,\ell})_k + t^*(w_{i-1,\ell})_k + (1-t^*)(v_{i+1,\ell})_k + t^*(w_{i+1,\ell})_k}{h^2}.\nnn\\\label{i1}
\eee
We consider the two cases in \R{i1}. First, we consider
$$\lambda + \frac{2\l[(1-t^*)(v_{i,\ell})_k + t^*(w_{i,\ell})_k\r]}{h^2} \le \frac{(1-t^*)(v_{i-1,\ell})_k + t^*(w_{i-1,\ell})_k + (1-t^*)(v_{i+1,\ell})_k + t^*(w_{i+1,\ell})_k}{h^2}$$
and after rearrangement we have
\bbb
\lambda & \le & \frac{1}{h^2}\l[(1-t^*)(v_{i-1,\ell})_k + t^*(w_{i-1,\ell})_k - 2\l[(1-t^*)(v_{i,\ell})_k + t^*(w_{i,\ell})_k\r] + (1-t^*)(v_{i+1,\ell})_k + t^*(w_{i+1,\ell})_k\r]\nnn\\
& = & \frac{1-t^*}{h^2}\l[(v_{i-1,\ell})_k- 2(v_{i,\ell})_k+(v_{i+1,\ell})_k\r] + \frac{t^*}{h^2}\l[(w_{i-1,\ell})_k- 2(w_{i,\ell})_k+(w_{i+1,\ell})_k\r]\nnn\\
& = & (1-t^*)\l(w_{xx}\bigg|_{(x_i,y_\ell,t_k)}+\OO\l(h^2\r)\r)+t^*\l(w_{xx}\bigg|_{(x_i,y_\ell,t_k)}+\OO\l(h^2\r)\r)\nnn\\
& = & w_{xx}\bigg|_{(x_i,y_\ell,t_k)}+\OO\l(h^2\r).\label{i2}
\eee
We now consider the other case of the inequality from \R{i1}
$$\frac{-\l[(1-t^*)(v_{i-1,\ell})_k + t^*(w_{i-1,\ell})_k + (1-t^*)(v_{i+1,\ell})_k + t^*(w_{i+1,\ell})_k\r]}{h^2} ~~~\le~~~ \lambda + \frac{2\l[(1-t^*)(v_{i,\ell})_k + t^*(w_{i,\ell})_k\r]}{h^2}$$
and by work similar to that used to obtain \R{i2}, we obtain
\bb{i3}
-\l(w_{xx}\bigg|_{(x_i,y_\ell,t_k)}+\OO\l(h^2\r)\r) \le \lambda.
\ee
Combining \R{i2} and \R{i3} gives
$$\|X_\ell\| \le |w_{xx}|~\bigg|_{(x_i,y_\ell,t_k)} + ch^2,$$
where $c$ is some positive constant independent of $h.$ Thus, we have
$$\|PD(\xx_k)\| \le \max_{\ell=1,\dots,N}\|X_\ell\| \le \max\l\{|(\w_{xx})_k|\r\}+ch^2.$$
Similar arguments give
$$\|RD(\xx_k)\| \le \max\l\{|(\w_{yy})_k|\r\}+ch^2,$$
and thus, combining the results, we have
$$\|PD(\xx_k)\|,\|RD(\xx_k)\| \le \max\l\{|\Delta \w_k|\r\}+ch^2,$$
as desired.
\end{proof}

\begin{lemma}
Assume that \R{cfl1} holds. Let $\w_k$ be the solution to \R{nonlinearmodel} and let $\vv_k$ be the solution to \R{factorednonlinear}. If $\xx_k = \vv_k + t^*(\w_k-\vv_k)$ for some $t^*\in(0,1)$ then we have
$$\|PD(R\xx_k)\| \le \max\{|\Delta^2\w_k|\}+ch^2,$$
where $D(R\xx_k) = \mbox{diag}(R\xx_k)$ and $c$ is a positive constant independent of $h.$
\end{lemma}

\begin{proof}
This is shown in a similar fashion as the previous lemma and is removed for brevity.

\end{proof}

\begin{thm}
Let $\tau_\ell,~\ell=0,1,\dots,k$ be sufficiently small. If
\bb{cfl2}
\frac{\tau_k}{h^2} ~~<~~ \frac{1}{4\max\{1,\max_{i=1,\dots,N^2}\{(\vv_k)_i,(\tilde{\vv}_k)_i\}\}}
\ee
holds, $\w_k\in W^{4,\infty}(\Omega),$ for $k\ge 0,$ where $\w_k$ is the true solution to \R{nonlinearmodel}, and $\|\f_v(\xx)\|\le K<\infty,$ for $\xx\in\mathbb{R}^{N^2},$ then the scheme \R{factorednonlinear} is unconditionally
stable in the von Neumann sense, that is,
$$\|\textbf{z}_{k+1}\| \leq \tilde{c} \|\textbf{z}_{0}\|,~~~k\geq 0, $$
where $\textbf{z}_0=\vv_0-\tilde{\vv}_0$ is an initial error,
$\textbf{z}_{k+1}=\vv_{k+1}-\tilde{\vv}_{k+1}$ is the $(k+1)$st perturbed
error vector, and $\tilde{c}>0$ is a constant independent of $k$ and $\tau_k.$
\end{thm}

\begin{proof}
Let $\vv_k$ be a solution and $\tilde{\vv}_k$ be a perturbed solution to \R{factorednonlinear}. We first note that if \R{cfl2} holds, then we have the results which follow from \R{cfl1} holding. Let
$$\Phi_k = \left(I - \frac{\tau_k}{2} P \right)\left(I - \frac{\tau_k}{2} R \right)\quad\mbox{and}\quad \Psi_k = \left(I + \frac{\tau_k}{2} P \right)\left(I + \frac{\tau_k}{2} R \right),$$
then we have
\bbb
\Phi_k\left(I - \frac{\tau_k}{2} P D(\vv_{k+1}) \right)\left(I - \frac{\tau_k}{2} R D(\vv_{k+1}) \right)\vv_{k+1} &=& \Psi_k\left(I + \frac{\tau_k}{2} P D(\vv_{k}) \right)\left(I + \frac{\tau_k}{2} R D(\vv_{k}) \right)\vv_{k}\nnn\\
&&+\frac{\tau_k}{2}(\f(\vv_{k+1})+\f(\vv_k)),\nnn\\
\label{a1}
\eee
and
\bbb
\Phi_k\left(I - \frac{\tau_k}{2} P D(\tilde{\vv}_{k+1}) \right)\left(I - \frac{\tau_k}{2} R D(\tilde{\vv}_{k+1}) \right)\tilde{\vv}_{k+1} &=& \Psi_k\left(I + \frac{\tau_k}{2} P D(\tilde{\vv}_{k}) \right)\left(I + \frac{\tau_k}{2} R D(\tilde{\vv}_{k}) \right)\tilde{\vv}_{k}\nnn\\
&&+\frac{\tau_k}{2}(\f(\tilde{\vv}_{k+1})+\f(\tilde{\vv}_k)),\nnn\\
\label{a2}
\eee
Letting $\z_k = \vv_k - \tilde{\vv}_k$ and subtracting \R{a2} from \R{a1}, we have
\bbb
\Phi_k\l(g(\vv_{k+1})-g(\tilde{\vv}_{k+1})\r) = \Psi_k\l(h(\vv_k)-h(\tilde{\vv}_k)\r)+\frac{\tau_k}{2}(\f(\vv_{k+1})-\f(\tilde{\vv}_{k+1}))+\frac{\tau_k}{2}(\f(\vv_{k})-\f(\tilde{\vv}_{k})),\label{a3}
\eee
where
$$g(\vv_{k+1}) = \l(I-\frac{\tau_k}{2}PD(\vv_{k+1})\r)\l(I-\frac{\tau_k}{2}RD(\vv_{k+1})\r)\vv_{k+1},\quad h(\vv_{k}) = \l(I+\frac{\tau_k}{2}PD(\vv_{k})\r)\l(I+\frac{\tau_k}{2}RD(\vv_{k})\r)\vv_{k}.$$
Note that, for $\tau_k$ sufficiently small we have $g(\vv_{k+1}) = g(\tilde{\vv}_{k+1}) + g_v(\xx_{k+1}^{(1)})\z_{k+1}$ for some $\xx_{k+1}^{(1)}\in\mathcal{L}(\vv_{k+1};\tilde{\vv}_{k+1}),$ where $\mathcal{L}(\vv_{k+1};\tilde{\vv}_{k+1})$ is the line segment connecting $\vv_{k+1}$ to $\tilde{\vv}_{k+1}$ in $\RR^{N^2}.$ Similarly we have that $h(\vv_k) = h(\tilde{\vv}_k) + h_v(\xx_k^{(2)})\z_k$ for some $\xx_k^{(2)}\in\mathcal{L}(\vv_k;\tilde{\vv}_k).$ Note that this means $\xx_{k+1}^{(1)} = \tilde{\vv}_{k+1} + t^*(\vv_{k+1} - \tilde{\vv}_{k+1})$ for some $t^*\in (0,1)$ and $\xx_k^{(2)} = \tilde{\vv}_k + s^*(\vv_k - \tilde{\vv}_k)$ for some $s^*\in(0,1).$ Thus, \R{a3} becomes
\bbb
\Phi_kg_v(\xx_{k+1}^{(1)})\z_{k+1} & = & \Psi_kh_v(\xx_k^{(2)})\z_k+\frac{\tau_k}{2}(\f(\vv_{k+1})-\f(\tilde{\vv}_{k+1}))+\frac{\tau_k}{2}(\f(\vv_{k})-\f(\tilde{\vv}_{k})).\label{a4}
\eee
Further, we have $\f(\vv_{k}) = \f(\tilde{\vv}_{k}) + \f_v(\xx_{k}^{(3)})\z_{k}$ for some $\xx_{k}^{(3)}\in\mathcal{L}(\vv_k;\tilde{\vv}_k).$ Thus, \R{a4} becomes
\bbb
\Phi_kg_v(\xx_{k+1}^{(1)})\z_{k+1} & = & \Psi_kh_v(\xx_k^{(2)})\z_k+\frac{\tau_k}{2}\f_v(\xx_{k+1}^{(3)})\z_{k+1}+\frac{\tau_k}{2}\f_v(\xx_k^{(3)})\z_k\nnn\\
\l(\Phi_kg_v(\xx_{k+1}^{(1)})-\frac{\tau_k}{2}\f_v(\xx_{k+1}^{(3)})\r)\z_{k+1} & = & \l(\Psi_kh_v(\xx_k^{(2)})+\frac{\tau_k}{2}\f_v(\xx_k^{(3)})\r)\z_k\nnn
\eee
and solving for $\z_{k+1}$ gives
\bbb
\z_{k+1} & = & \l(\Phi_kg_v(\xx_{k+1}^{(1)})-\frac{\tau_k}{2}\f_v(\xx_{k+1}^{(3)})\r)^{-1}\l(\Psi_kh_v(\xx_k^{(2)})+\frac{\tau_k}{2}\f_v(\xx_k^{(3)})\r)\z_k\nnn\\
& = & \l(g_v(\xx_{k+1}^{(1)})\r)^{-1}\Phi_k^{-1}\l(I-\frac{\tau_k}{2}\l(g_v(\xx_{k+1}^{(1)})\r)^{-1}\Phi_k^{-1}\f_v(\xx_{k+1}^{(3)})\r)^{-1}\nnn\\
&&~~~~~\times\l(\Psi_kh_v(\xx_k^{(2)})\r)\l(I+\frac{\tau_k}{2}\l(h_v(\xx_k^{(2)})\r)^{-1}\Psi_k^{-1}\f_v(\xx_k^{(3)})\r)\z_k.\label{a6}
\eee
Taking the norm of both sides of \R{a6} gives
\bbb
\|\z_{k+1}\| & = & \l\|\l(g_v(\xx_{k+1}^{(1)})\r)^{-1}\Phi_k^{-1}\l(I-\frac{\tau_k}{2}\l(g_v(\xx_{k+1}^{(1)})\r)^{-1}\Phi_k^{-1}\f_v(\xx_{k+1}^{(3)})\r)^{-1}\r.\nnn\\
&&~~~~~\l.\times\l(\Psi_kh_v(\xx_k^{(2)})\r)\l(I+\frac{\tau_k}{2}\l(h_v(\xx_k^{(2)})\r)^{-1}\Psi_k^{-1}\f_v(\xx_k^{(3)})\r)\z_k\r\|\nnn\\
& \le & \l\|\l(g_v(\xx_{k+1}^{(1)})\r)^{-1}\r\|\l\|\Phi_k^{-1}\r\|\l\|\l(I-\frac{\tau_k}{2}\l(g_v(\xx_{k+1}^{(1)})\r)^{-1}\Phi_k^{-1}\f_v(\xx_{k+1}^{(3)})\r)^{-1}\r\|\nnn\\
&&~~~~~\times\l\|\Psi_k\r\|\l\|h_v(\xx_k^{(2)})\r\|\l\|\l(I+\frac{\tau_k}{2}\l(h_v(\xx_k^{(2)})\r)^{-1}\Psi_k^{-1}\f_v(\xx_k^{(3)})\r)\r\|\|\z_k\|\nnn\\
& \le & \l(1+c_{1,k}\tau_k\r)\l\|\l(g_v(\xx_{k+1}^{(1)})\r)^{-1}\r\|\l\|\l(I-\frac{\tau_k}{2}\l(g_v(\xx_{k+1}^{(1)})\r)^{-1}\Phi_k^{-1}\f_v(\xx_{k+1}^{(3)})\r)^{-1}\r\|\nnn\\
&&~~~~~\times\l\|h_v(\xx_k^{(2)})\r\|\l\|\l(I+\frac{\tau_k}{2}\l(h_v(\xx_k^{(2)})\r)^{-1}\Psi_k^{-1}\f_v(\xx_k^{(3)})\r)\r\|\|\z_k\|,\label{a7}
\eee
by Lemma 3.4. We now show bounds for $\l\|\l(g_v(\xx_{k+1}^{(1)})\r)^{-1}\r\|$ and $\l\|h_v(\xx_k^{(2)})\r\|.$ By carefully considering the matrix-vector product and then differentiating, we observe that
\bbb
g_v(\xx_{k+1}^{(1)}) & = & \frac{\partial}{\partial \vv}\l(I-\frac{\tau_k}{2}PD(\vv)\r)\l(I-\frac{\tau_k}{2}RD(\vv)\r)\vv\bigg|_{\vv=\xx_{k+1}^{(1)}}\nnn\\
& = & I - \tau_kPD(\xx_{k+1}^{(1)}) - \tau_kRD(\xx_{k+1}^{(1)}) + \frac{\tau_k^2}{2}PD(\xx_{k+1}^{(1)})RD(\xx_{k+1}) + \frac{\tau_k^2}{4}PD(R\xx_{k+1}^{(1)})\nnn\\
& = & \l(I-\tau_kPD(\xx_{k+1}^{(1)})\r)\l(I-\tau_kRD(\xx_{k+1}^{(1)})\r) + \OO\l(\tau_k^2\r),\label{a8}
\eee
for $\tau_k$ sufficiently small, where we have appealed to the fact that $\|PD(\xx_{k+1}^{(1)})\|,\|RD(\xx_{k+1}^{(1)})\| < \infty$ by Lemma 3.5. and $\|PD(R\xx_{k+1}^{(1)})\|<\infty$ by Lemma 3.6, combined with our assumptions that $\w_k\in W^{4,\infty}$ for $k\ge 0.$ Using \R{a8} we have
$$\l(g_v(\xx_{k+1}\r)^{-1} = \l(I-\tau_kRD(\xx_{k+1})\r)^{-1}\l(I-\tau_kPD(\xx_{k+1})\r)^{-1} + \OO\l(\tau_k^2\r)$$
and then
\bbb
\l\|\l(g_v(\xx_{k+1}^{(1)})\r)^{-1}\r\| & = & \l\|\l(I-\tau_kRD(\xx_{k+1})\r)^{-1}\l(I-\tau_kPD(\xx_{k+1})\r)^{-1} + \OO\l(\tau_k^2\r)\r\|\nnn\\
& = & \l\|E(\tau_kPD(\xx_{k+1})E(\tau_kRD(\xx_{k+1}) + \OO\l(\tau_k^2\r)\r\|\nnn\\
& \le & E(\tau_k\mu(PD(\xx_{k+1}))E(\tau_k\mu(RD(\xx_{k+1})) + c_{4,k}\tau_k^2\nnn\\
& \le & 1 + c_{2,k}\tau_k,\label{a9}
\eee
where $c_{2,k},~c_{4,k}$ are positive constants independent of $h,~\tau_k,~k,$ by similar arguments to that used in Lemma 3.4. By similar arguments used to obtain \R{a8}, we may obtain
$$h_v(\xx_{k+1}^{(2)})=\l(I+\tau_kPD(\xx_k^{(2)})\r)\l(I+\tau_kRD(\xx_k^{(2)})\r)+\OO\l(\tau_k^2\r)$$
and using similar arguments as those used to obtain \R{a8} we may obtain
\bb{a9}
\l\|h_v(\xx_{k+1}^{(2)})\r\| \le 1 + c_{3,k}\tau_k,
\ee
where $c_{3,k}$ is a positive constant independent of $h,~\tau_k,~\mbox{and}~ k.$ Combining
\R{a7}-\R{a9} we now have
\bbbb
\|\z_{k+1}\| &\le& \prod_{i=1}^3\l(1+c_{i,k}\tau_k\r)\l\|\l(I-\frac{\tau_k}{2}\l(g_v(\xx_{k+1}^{(1)})\r)^{-1}\Phi_k^{-1}\f_v(\xx_{k+1}^{(3)})\r)^{-1}\r\|\l\|\l(I+\frac{\tau_k}{2}\l(h_v(\xx_k^{(2)})\r)^{-1}\Psi_k^{-1}\f_v(\xx_k^{(3)})\r)\r\|\|\z_k\|\\
&\le& \l(1+c_{5,k}\tau_k\r)\l\|\l(I-\frac{\tau_k}{2}\l(g_v(\xx_{k+1}^{(1)})\r)^{-1}\Phi_k^{-1}\f_v(\xx_{k+1}^{(3)})\r)^{-1}\r\|\l\|\l(I+\frac{\tau_k}{2}\l(h_v(\xx_k^{(2)})\r)^{-1}\Psi_k^{-1}\f_v(\xx_k^{(3)})\r)\r\|\|\z_k\|.
\eeee
For $\tau_k$ sufficiently small we have
\bbb
\|\z_{k+1}\| & \le & \l(1+c_{5,k}\tau_k\r)\l\|E\l(\frac{\tau_k}{2}\l(g_v(\xx_{k+1}^{(1)})\r)^{-1}\Phi_k^{-1}\f_v(\xx_{k+1}^{(3)})\r)+\OO\l(\tau_k^2\r)\r\|\nnn\\
&&~~~\times\l\|E\l(\frac{\tau_k}{2}\l(h_v(\xx_k^{(2)})\r)^{-1}\Psi_k^{-1}\f_v(\xx_k^{(3)})\r)+\OO\l(\tau_k^2\r)\r\|\|\z_k\|\nnn\\
& \le & \l(1+c_{5,k}\tau_k\r)\l[E\l(\frac{\tau_k}{2}\l[ \l\| \l(g_v(\xx_{k+1}^{(1)})\r)^{-1} \r\|\l\|\Phi_k^{-1}\r\| \l\|\f_v(\xx_{k+1}^{(3)})\r\|+c_{6,k}\tau_k^2\r.\r.\r.\nnn\\
&&~~~\l.\l.\l. +\l\|\l(h_v(\xx_k^{(2)})\r)^{-1}\r\|\l\| \Psi_k^{-1}\r\| \l\|\f_v(\xx_k^{(3)})\r\|\r]\r)+c_{7,k}\tau_k^2\r]\|\z_k\|\nnn\\
& \le & \l(1+c_{5,k}\tau_k\r)\l[E\l(\frac{\tau_kK}{2}\l[ \l\| \l(g_v(\xx_{k+1}^{(1)})\r)^{-1} \r\|\l\|\Phi_k^{-1}\r\| +\l\|\l(h_v(\xx_k^{(2)})\r)^{-1}\r\|\l\| \Psi_k^{-1}\r\|\r]\r)+c_{8,k}\tau_k^2\r]\|\z_k\|\nnn\\
& \le & \l(1+c_{5,k}\tau_k\r)\l[E\l(\frac{\tau_kK}{2}\l(2+c_{9,k}\tau_k\r)\r)+c_{8,k}\tau_k^2\r]\|\z_k\| ~~\le~~ \l(1+c_k\tau_k\r)\|\z_k\|,\label{a10}
\eee
where $c_{i,k},~i=5,6,7,8,9,~c_k$ are positive constants independent of $h,~\tau_k,$ and $k.$ Considering \R{a10}, recursively, we have
$$\|\z_{k+1}\| ~~\le~~ \prod_{i=0}^k\l(1+c_i\tau_i\r)\|\z_0\| ~~\le~~ \l(1 + C\sum_{i=0}^k\tau_i\r)\|\z_0\| ~~\le~~ \l(1 + CT\r)\|\z_0\| ~~\le~~ \tilde{c}\|\z_0\|,$$
where $c_i,~i=0,\dots,k,~C,~\tilde{c}$ are constants independent of $h,~\tau_i,~i=0,\dots,k,$ and the step $k,$ and $\sum_{i=0}^k\tau_i \le T < \infty.$ This gives the desired stability.
\end{proof}

\begin{rem}
It should be noted that since $\xx_k^{(i)}\in\mathcal{L}(\vv_k;\tilde{\vv}_k),~i=1,2,3,$ then we have that any matrices involving $\xx_k^{(i)}$ will also satisfy \R{cfl2}, hence the positivity of our solution is not affected by the previous analysis. Further, Corollary 3.1 will apply to the above.
\end{rem}


In the above theorem, stability has been guaranteed without freezing any of the nonlinear terms. We now show that the conditions used to ensure stability are also sufficient to ensure convergence. This current analysis is an improvement as it considers the contributions of all terms.

\begin{thm}
Let $\tau_\ell,~\ell=0,1,\dots,k$ be sufficiently small. If
\bb{cfl3}
\frac{\tau_k}{h^2} ~~<~~ \frac{1}{4\max\{1,\max_{i=1,\dots,N^2}\{(\vv_k)_i,(\w_k)_i\}\}}
\ee
holds, $\w_k\in W^{4,\infty}(\Omega),$ for $k\ge 0,$ where $\w_k$ is the true solution to \R{nonlinearmodel}, and $\|\f_v(\xx)\|\le K<\infty,$ for $\xx\in\mathbb{R}^{N^2},$ then \R{factorednonlinear} is convergent.
\end{thm}

\begin{proof}

Let $\w_k$ be the true solution to \R{nonlinearmodel}. Once again, we note that if \R{cfl3} holds, then we have the results which follow from \R{cfl1} holding. Then, similar to Theorem 3.3, we have
\bbb
\Phi_k\left(I - \frac{\tau_k}{2} P D(\w_{k+1}) \right)\left(I - \frac{\tau_k}{2} R D(\w_{k+1}) \right)\w_{k+1} &=& \Psi_k\left(I + \frac{\tau_k}{2} P D(\w_{k}) \right)\left(I + \frac{\tau_k}{2} R D(\w_{k}) \right)\w_{k}\nnn\\
&&+\frac{\tau_k}{2}(\f_{k+1}+\f_k)+\OO\l(\tau^3_k\r)+\OO\l(\tau_kh^2\r),\nnn\\
\label{c1}
\eee
where $D(\w_k) = \mbox{diag}(\w_k).$ Letting $\e_k = \w_k - \vv_k$ and subtracting \R{factorednonlinear} from \R{c1}, we have
\bbb
\Phi_k\l(g(\w_{k+1})-g(\vv_{k+1})\r) = \Psi_k\l(h(\w_k)-h(\vv_k)\r) + \OO\l(\tau_k^3\r)+\OO\l(\tau_kh^2\r),\label{c2}
\eee
where $g$ and $h$ are defined as in Theorem 3.3.
Note that, for $\tau_k$ sufficiently small we have $g(\w_{k+1}) = g(\vv_{k+1}) + g_v(\xx_{k+1}^{(1)})\e_{k+1}$ for some $\xx_{k+1}^{(1)}\in\mathcal{L}(\w_{k+1};\vv_{k+1}),$ $h(\w_k) = h(\vv_k) + h_v(\xx_k^{(2)})\e_k$ for some $\xx_k^{(2)}\in\mathcal{L}(\w_k;\vv_k),$ and $\f(\w_k) = \f(\vv_k) + \f_v(\xx_k^{(3)})\e_k$ for some $\xx_k^{(3)}\in\mathcal{L}(\w_k;\vv_k).$ Thus, \R{c2} becomes
\bbb
\e_{k+1} & = & \l(g_v(\xx_{k+1}^{(1)})\r)^{-1}\Phi_k^{-1}\l(I-\frac{\tau_k}{2}\l(g_v(\xx_{k+1}^{(1)})\r)^{-1}\Phi_k^{-1}\f_v(\xx_{k+1}^{(3)})\r)^{-1}\l(\Psi_kh_v(\xx_k^{(2)})\r)\nnn\\
&&~~~~~\times\l(I+\frac{\tau_k}{2}\l(h_v(\xx_k^{(2)})\r)^{-1}\Psi_k^{-1}\f_v(\xx_k^{(3)})\r)\e_k + \OO\l(\tau_k^3\r) + \OO\l(\tau_kh^2\r)\label{c3}
\eee
by similar arguments as those used in Theorem 3.3 to obtain \R{a6}. Taking the norm of both sides of \R{c3} gives
\bbb
\|\e_{k+1}\| & \le & \l(1+c_k\tau_k\r)\|\e_k\| + c_{1,k}\tau_k^3 + c_{2,k}\tau_kh^2,\label{c4}
\eee
where $c_{1,k},~c_{2,k},~c_k$ are positive constants independent of $h,~\tau_k,~k,$
for $\tau_k$ sufficiently small and we have appealed to the fact that $\|PD(\xx_{k+1}^{(1)})\|,\|RD(\xx_{k+1}^{(1)})\| < \infty$ by Lemma 3.5. and $\|PD(R\xx_{k+1}^{(1)})\|<\infty$ by Lemma 3.6, combined with our assumptions that $\w_k\in W^{4,\infty}$ for $k\ge 0.$
Considering \R{c4} recursively gives
\bbb
\|\e_{k+1}\|
& \le & \prod_{i=0}^k \l(1+c_i\tau_i\r)\|\e_0\| + \sum_{i=0}^k c_i\l(\tau_i^3+\tau_ih^2\r) ~~\le~~ \l(1 + C\sum_{i=0}^k\tau_i\r)\|\e_0\| + C\sum_{i=0}^k\l(\tau_i^3+\tau_ih^2\r)\nnn\\
& \le & \l(1+ CT\r)\|\e_0\| + C\sum_{i=0}^k\l(\tau_i^3+\tau_ih^2\r) ~~\le~~ C\sum_{i=0}^k\l(\tau_i^3+\tau_ih^2\r),\label{ccc1}
\eee
where $C$ is a positive constant independent of $h,~\tau_k,~\mbox{and}~ k,$ $\sum_{i=0}^k\tau_i \le T,$ and we have used the fact that $\e_0 = \boldsymbol{0}.$ Also note that
$$\sum_{i=0}^k \tau_i^3 \le \tau^2\sum_{i=0}^k\tau_i \le \tau T\quad\quad\mbox{and}\quad\quad \sum_{i=0}^k\tau_ih^2 = h^2\sum_{i=0}^k \tau_i \le Th^2,$$
where $\max_{i=0,1,\dots,k}\{\tau_i\} \le \tau.$
Therefore,
$$\lim_{\tau,h\to 0}\|\e_{k+1}\| = \lim_{\tau,h\to 0} C\l(T\tau + Th^2\r) = 0,$$
which ensures the anticipated convergence.
\end{proof}

\section{\label{Sec:Experiments} Numerical Experiments}  \clearallnum

We provide numerical experiments that provide empirical evidence which suggests our nonlinear operator splitting method is stable, convergent, and efficient.  The first two examples focus exclusively on simulations of \R{nonlinearmodel}.  The first example examines the numerical convergence rate and computational efficiency, in light of a known exact solution.  The second example examines the numerical solution in the case of no reactive term.  In this latter case no theoretical solution is known, however, the numerical solution satisfies our energy estimates and converges at the anticipated rate.  In our last two examples, the numerical procedure is used to examine the effect of self-diffusion on the nonlinear food chain model in \R{eq:(1.1)}-\R{eq:(1.3)}.

The computations are carried out on a Matlab\textsuperscript{\textregistered} platform and its parallel computing toolbox on a HP EliteDesk 800 G1 work station with an Intel\textsuperscript{\textregistered} Core(TM) i7-4770 3.40GHz processor with 16 GB of RAM.

\subsection{Nonlinear Model - Example 1}

Consider \R{nonlinearmodel} with
$$f(x,y,t) = -2 \pi^2 \l[ \l(\cos(\pi x) \sin(\pi y)\r)^2 + \l(\cos(\pi y) \sin(\pi x)\r)^2 - 2 \l(\sin(\pi x) \sin(\pi y)\r)^2 \r] \exp\l(-4 \pi^2 t\r)$$
This reactive term has been chosen so that an exact solution to the partial differential equation is known, that is,
$$ u(x,y,t) = \sin(\pi x) \sin(\pi y) \exp(-2 \pi^2 t).$$
Let $u_{i,j}^{(\tau)}$ be the numerical solution at $x_i=i h$, $y_j= j h$, and at time $T$ with $\tau$ as the temporal size.  For a fixed $h$ we have $|u_{i,j}^{(\tau)} - u| \approx C \tau^p$ for which $p$ is the order of accuracy that can be estimated as
\[ p \approx \frac{1}{\ln 2} \frac{1}{N^2} \sum_{i,j=1}^N \ln \frac{\l| u_{i,j}^{(\tau)} - u(x_i,y_j,T)\r|}{\l| u_{i,j}^{(\tau/2)} - u(x_i,y_j,T)\r|}. \]
Let $h=.01$ and based on \R{cfl} let $\tau = \frac12 10^{-4}$.  Let $T=1$. We estimate the approximate order of
$$p\approx 1.998841.$$
This indicates that our nonlinear splitting algorithm is converging at the anticipated second order rate.

Since the exact solution decays exponentially the temporal step is constant throughout the entire computation.  Clearly, the maximum of the theoretical solution occurs at $(\frac12,\frac12)$.  Hence, we calculate the natural logarithm of the numerical and exact solution over time at the maximum location. Therefore, we expect to see a linear function with a slope of $-2 \pi^2$. In Figure \ref{Example1Plot}(a) we see that the numerical solution decays at nearly the identical rate as the exact solution; the two curves are virtually indistinguishable.  We use a linear least squares to estimate the slope of the numerical solution decay and determine a value of $-19.7392088$.  In Figure \ref{Example1Plot}(b) the absolute difference is shown to aid in the comparison. Notice that the slope of \ref{Example1Plot}(b) is $2\pi^2 - 19.7392088\approx .006.$

\begin{figure}[h]
\begin{center}
\includegraphics[scale=.5]{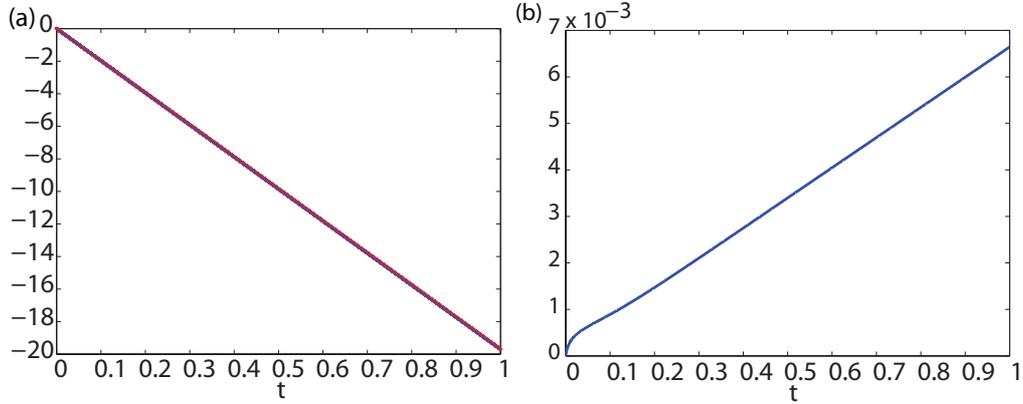}
\caption{The maximum value of the numerical and theoretical solution both occur at the center of the domain $(.5,.5)$. (a) A plot of the natural logarithm of $u(.5,.5,t_k)$ (blue) and $u_{50,50,k}$ (red) are shown.  The slope is calculated through a linear least squares and is approximately $-19.7392088,$ which is close to the anticipated exponential decay rate of $-2\pi^2\approx -19.7392088$. (b) An absolute difference of the values shown in (a).  The slope of the difference of the logs is $.006$. This shows strong agreement between the theoretical and numerical solutions. Parameters used: $h=.01$, $\tau = 5\times 10^{-5}$.}
\label{Example1Plot}
\end{center}
\end{figure}

We shall examine the computationally efficiency of the algorithm as we increase the number of unknowns, $N$, that is, as $h$ decreases in size.  Let $\tau = 10^{-6}$ and consider the computational time for $1000$ temporal steps for $N=21,31, \ldots, 401$.  The computational time, in seconds, is determine for increasing $N$. Figure \ref{Example1CompTime} shows a log-log plot of the computational time versus $N$.  A linear least squares approximation is used to determine the slope of the line and is found to be $1.654628$. Hence, the computational time scales as $N^{1.654628}$, which shows the proposed nonlinear splitting method is computationally efficient.

\begin{figure}[h]
\begin{center}
\includegraphics[scale=.4]{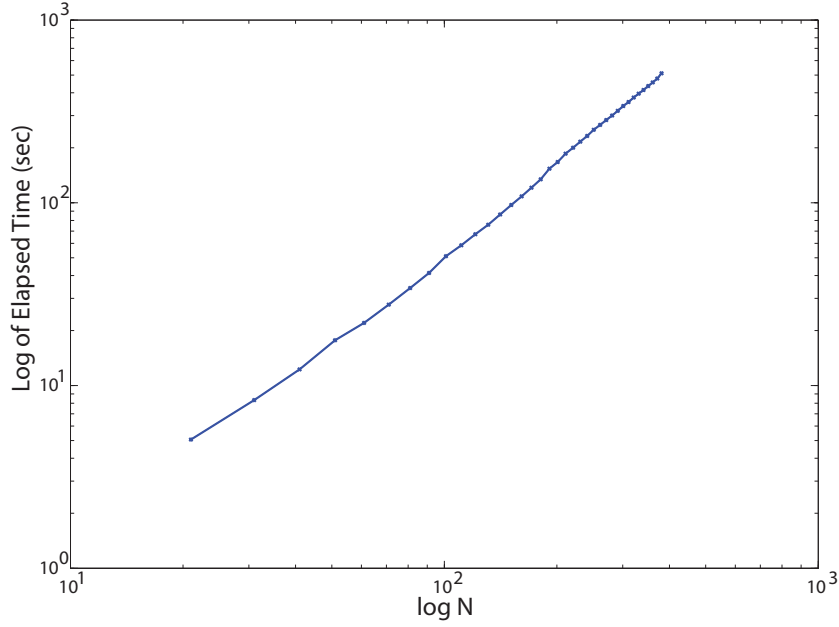}
\caption{A log-log plot of the computational time, in seconds, versus $N$ after $1000$ iterations.  The temporal step is held constant, $\tau = 10^{-6}$, while $h=1/(N-1)$.  A linear least squares approximates the slope of the line to be $1.654628$. This indicates that the computational time is scales as $N^{1.654628}$.  The computational time of an efficient scheme should scale no slower than $N^2$.  This indicates that the proposed nonlinear splitting scheme is highly efficient.}
\label{Example1CompTime}
\end{center}
\end{figure}

\subsection{Nonlinear Model - Example 2}

Consider \R{nonlinearmodel} with $f(x,y,t)=0$.  Let
$$ E(t) = \| u \|_2^2 \equiv  \int_{\Omega} | u(x,y,t) |^2 d\Omega.$$
Multiplying \R{nonlinearmodel} by $u$ and integrating over the spatial domain yields, after using the divergence theorem,
$$ E'(t) = 2 \int_{\Omega} (u + u^2)\Delta u d\Omega.$$
If we assume Dirichlet boundary conditions and then integrate by parts, we obtain
$$ \int_{\Omega} u^2 \Delta u d\Omega = - 2 \int u | \nabla u |^2 d\Omega  $$
Hence,
$$ E'(t) +  2 \int_{\Omega} |\nabla u|^2 d\Omega \leq 0.$$
By the Poincar\'{e} inequality we have,
$$ \frac{1}{C} \| u \|_2^2 \leq \int_{\Omega} |\nabla u|^2 d\Omega $$
for a constant $C$.  Therefore,
$$ E'(t) +  \frac{2}{C}E(t) \leq 0.$$
Hence, we have an upper bound for the energy norm for $u$, that is,
$$ E(t) \leq \exp\l(-\frac{2}{C}t\r)E(0).$$
Let $h=.01$, $\tau=10^{-5}$, and $T=5$.  Let the initial condition be $u(x,y,0)=\sin(\pi x) \sin(\pi y).$ We determine the numerical solution and determine the energy norm over the duration of the computation.  In Figure \ref{Example2LogEnergy} we plot the natural logarithm of energy norm and clearly see that the numerical solution decays at a rate of approximately $-19.7397958.$ This further validates the effectiveness of the numerical approximation technique.

\begin{figure}[h]
\begin{center}
\includegraphics[scale=.4]{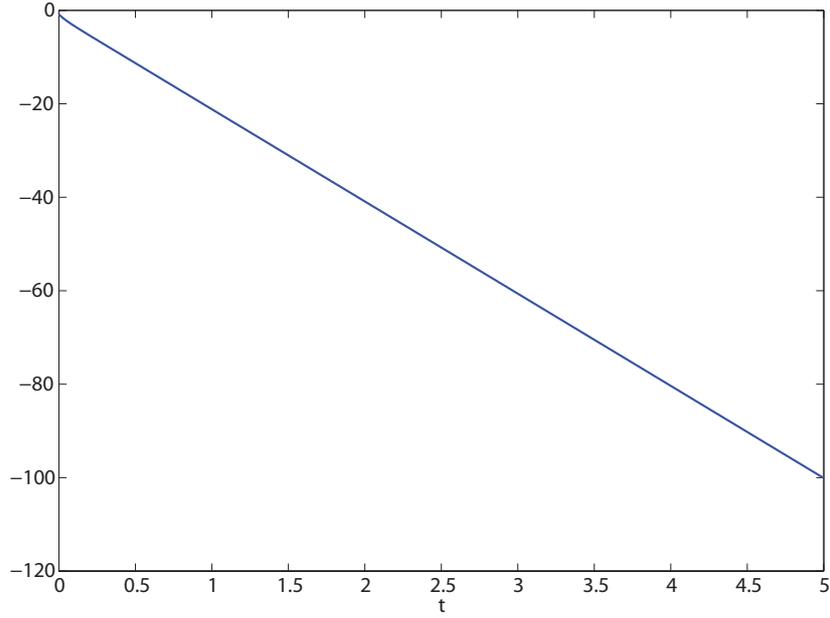}
\caption{A plot of the natural logarithm of the energy norm of $v(x,y,t)$ throughout the duration of the computation.  The slope is determined to be $-19.7397958$ using a linear least squares approximation.  This agrees with the anticipated upper bound and suggests a value of $C \approx -.101318$. The temporal step is held constant, $\tau = 10^{-6}$, while $h=.01$.}
\label{Example2LogEnergy}
\end{center}
\end{figure}

Consider the same initial condition and spatial step size.  We approximate the temporal convergence rate in a similar fashion to the previous example, that is,
\[ p \approx \frac{1}{\ln 2} \frac{1}{N^2} \sum_{i,j=1}^N \ln \frac{\l| u_{i,j}^{(\tau)} - u_{i,j}^{(\tau_f)}\r|}{\l| u_{i,j}^{(\tau/2)} - u_{i,j}^{(\tau_f)}\r|}, \]
where $\tau=\frac12 10^{-4}$ and $u_{i,j}^{(\tau_f)}$ is the numerical approximation using a fine temporal step size, $\tau_f = 10^{-12}$. We find that
$$ p \approx  1.994504,$$
which is at the predicated second order convergence rate.

\subsection{Three-Species Food Chain Model}

The nonlinear food-chain model is approximated through our nonlinear operator splitting scheme. Recall Eqs.~\ref{eq:(1.1)}-\ref{eq:(1.3)}, namely,

\begin{eqnarray*}
\partial _{t}r &=& d_3 \Delta r + d_4 \Delta r^2 + cr^{2}-w _{3}\frac{r^{2}}{v+D_{3}} \\
\partial _{t}v &=& d_{2}\Delta v-a_{2}v+w_{1}\frac{uv}{u+D_{1}}-w_{2}\frac{vr}{v+D_{2}} \\
\partial _{t}u &=& d_{1}\Delta u+a_{1}u-b_{2}u^{2}-w_{0}\frac{uv}{u+D_{0}}.
\end{eqnarray*}
The partial differential equation for $r$ is determined through our nonlinear splitting scheme, while the other partial differential equations are determined through a second order accurate Peaceman-Rachford splitting scheme \cite{Beauregard2013}. For clarity, the numerical procedure is detailed here:

\begin{subequations}
\begin{eqnarray*}
   \l(I - \frac{\tau_k}{2} P\r)\tilde{\textbf{r}}^{(1)} &=& \left(I + \frac{\tau_k}{2} \left(P + 2R + 2PD_k + 2RD_k\right)\right) \textbf{r}_k + \tau_k \textbf{h}_k, \\
\l(I - \frac{\tau_k}{2} P\r)\tilde{\textbf{v}}^{(1)} &=& \left(I + \frac{\tau_k}{2} P\right) \textbf{v}_k + \frac12 \tau_k \textbf{g}_k \\
\l(I - \frac{\tau_k}{2} P\r)\tilde{\textbf{u}}^{(1)} &=& \left(I + \frac{\tau_k}{2} P\right) \textbf{u}_k + \frac12 \tau_k \textbf{f}_k \\
 \l(I - \frac{\tau_k}{2} R\r)\tilde{\textbf{r}}^{(2)} &=& \tilde{\textbf{r}}^{(1)} - \frac{\tau_k}{2} R \textbf{r}_k,\\
\l(I - \frac{\tau_k}{2} R\r)\tilde{\textbf{v}}_{k+1} &=& \left(I + \frac{\tau_k}{2} R\right) \textbf{v}^{(1)} + \frac12 \tau_k \textbf{g}_{k+1} \\
\l(I - \frac{\tau_k}{2} R\r)\tilde{\textbf{u}}_{k+1} &=& \left(I + \frac{\tau_k}{2} R\right) \textbf{u}^{(1)} + \frac12 \tau_k \textbf{f}_{k+1} \\
 \l(I - \frac{\tau_k}{2} PD_{k+1}\r)\tilde{\textbf{r}}^{(3)} &=& \tilde{\textbf{r}}^{(2)} - \frac{\tau_k}{2} P D_k \textbf{r}_k,\\
 \l(I - \frac{\tau_k}{2} RD_{k+1}\r)\textbf{r}_{k+1} &=& \tilde{\textbf{r}}^{(3)} - \frac{\tau_k}{2} R D_k \textbf{r}_k + \frac{\tau_k}{2}\left( \textbf{h}_{k+1} - \textbf{h}_k\right),
\end{eqnarray*}
\end{subequations}
where $\textbf{r}_k = \left(r_{11}^k, r_{21}^k, \ldots, r_{NN}^k\right)^{\top},~ \textbf{v}_k = \left(v_{11}^k, v_{21}^k, \ldots, v_{NN}^k\right)^{\top},$ and $\textbf{u}_k = \left(u_{11}^k, u_{21}^k, \ldots, u_{NN}^k\right)^{\top}$.

\subsubsection{Example 3}

 As in the previous example, the convergence rate is estimated.  Let $u(x,y,0)=v(x,y,0)=r(x,y,0)=\sin(\pi x) \sin(\pi y)$, $d_1=d_2=d_3=d_4=a_1=a_2=w_1=w_2=w_3=1$, $D_0=D_1=D_2=D_3=10$, and $c = .2$.  Let $h=.01$ and $\tau$ is fixed at $10^{-4}$.  We determine the convergence rate, $p$, is approximately $2.001871$, $1.991905$, and $1.994640$ for the numerical solutions for $u$, $v$, and $r$, respectively.  This indicates the numerical solution is converging at the anticipated second order rate.

As in example 1, we estimate the computational efficiency by determining the computational time to complete $1000$ temporal steps for $N=21,31, \ldots, 401$.  The computational time, in seconds, is determined for increasing $N$. It is determined that the computational time scales as $N^{1.736}$, which indicates our method is computationally efficient.

\subsubsection{Example 4}

Consider the case when $d_4=0$.  The population $r$ may blow-up in finite time, that is, $\displaystyle \lim_{t\rightarrow T^-<\infty} r(x,y,t) \rightarrow \infty$.  Notice the coefficient on $r^2$ is $c - w_3/(v + D_3)$.  If $v=0$ and $c - w_3/D_3 <0$ it is known that $r(x,y,t)$ will decay to zero.  However, if $v\neq 0$ then it has been shown that $r$ may still blow-up in finite time even if $c - w_3/D_3 <0$ for Neumann boundary conditions \cite{AA02}.  In the case of Dirichlet boundary conditions it has been conjectured that for a fixed parameter set and initial conditions there will exists a critical $c^*$ for which $\forall c>c^*$ the solution will blow-up in finite time.  We provide empirical evidence that supports this conjecture for the parameters shown in Table \ref{tab:table2}.  The initial conditions are
\[ v(x,y,0)=r(x,y,0)=100 \sin(\pi x) \sin(\pi y),  ~~ u(x,y,0) = .1 v(x,y,0) \]
Let $h=.01$ and $\tau_0=5\times 10^{-6}$.  The minimum stepsize is set to be $10^{-10}$.  In the case of $d_4=0$ the critical $c^*\approx 5.07$.  For $d_4=.025$ the $c^*\approx 5.25$. Hence, we see that for $d_4\neq 0$ the critical value of $c^*$ is larger, due to the fact the nonlinearity $(\Delta r^2)$ further dampens the population growth of $r$.  

\begin{table}[h!]
  \begin{center}
    \caption{List of parameters used in Example 4.}
    \label{tab:table2}
    \begin{tabular}{llll}\hline
$a_1 = 5.0$ \quad \quad \quad \quad & $a_2 = 0.75$ \quad \quad \quad \quad & $w_{0} = 0.55$ \quad \quad \quad \quad & $w_{1} = 1.0$ \\
$w_{2} =0.25$\ \quad \quad \quad \quad  & $w_{3} = 1.2$  \quad \quad \quad \quad  &$b_{2} = 0.5$  \quad \quad \quad \quad &$D_0 = 20.0$  \\
$D_1= 13.0$  \quad \quad \quad \quad & $D_2 = 10.0$ \quad \quad \quad \quad & $D_3 = 20.0$ \quad \quad \quad \quad & \\
$d_1 = 0.1$\quad \quad \quad \quad & $d_2 = 0.1$\quad \quad \quad \quad &$d_3=0.1$\\ \hline
    \end{tabular}
  \end{center}
\end{table}



\section{\label{Sec:Conclusions} Conclusions}  \clearallnum

In this paper a new adaptive nonlinear operator splitting scheme was developed and analyzed to solve reaction diffusion equations with a nonlinear self-diffusion term.  Under minimal criteria we are able to shown convergence and stability of the proposed method.  The numerical experiments further validate these results. The method is also computationally efficient and experiments suggest the computational time to completion scales less than $N^2$, which is the total number of unknowns.  The efficiency of the algorithm is a notable accomplishment of the splitting design while the numerical analysis provides confidence in the utility of the numerical algorithm in applications, in particular nonlinear food chain models.  In particular, our examples indicated that our method is computationally efficient and accurate for a complicated and nonlinear food chain models. 

\section*{Acknowledgements}

The first author would like to explicitly express his gratitude for an internal research grant (No. 107552-26423-150) from Stephen F. Austin State University.


\end{document}